\documentclass[12pt,authoryear,letterpaper,reqno,oneside]{amsart}
\pdfoutput=1

\usepackage[left=1.5in,right=1.5in,bottom=1.5in,top=1.5in]{geometry}

\usepackage{amsmath,amsfonts,amssymb}%
\usepackage{mathrsfs}
\usepackage[usenames]{color}
\usepackage{setspace}
\usepackage{microtype}
\usepackage{subfiles}

\theoremstyle{plain}
\newtheorem{proposition}{Proposition}[section]
\newtheorem{lemma}[proposition]{Lemma}
\newtheorem{corollary}[proposition]{Corollary}

\newtheorem{theorem}[proposition]{Theorem}

\theoremstyle{definition}
\newtheorem{definition}[proposition]{Definition}
\newtheorem*{construction}{Construction}

\theoremstyle{remark}

\newcommand{\cut}{\square}

\newcommand{\mynewpage}{\newpage}

\usepackage{graphicx}\graphicspath{{figures/}}

\definecolor{LightGray}{rgb}{.8,.8,.8}

\definecolor{darkred}{rgb}{0.5,0,0}
\definecolor{darkgreen}{rgb}{0, 0.3,0}
\definecolor{darkblue}{rgb}{0,0,0.6}

\usepackage[colorlinks,citecolor=darkblue,urlcolor=darkblue,linkcolor=darkgreen]{hyperref}
\usepackage{hypernat}

\newcommand{\<}{\langle}
\renewcommand{\>}{\rangle}

\newcommand{\ER}{Erd\H{o}s--R\'enyi}

\newcommand{\G}{{\mathbb{G}}}

\newcommand{\Graphs}{{\mathcal{G}}}

\newcommand{\Graphons}{\mathcal{W}_0}
\newcommand{\StepGraphons}{\mathcal{S}}
\newcommand{\StepGraphonsRF}{\mathcal{S}_0}
\newcommand{\SampledFiniteGraphs}{\mathcal{S}_0^*}
\newcommand{\ExchMeasures}{\mathcal{E}}

\newcommand{\floor}[1]{\lfloor #1 \rfloor}

\makeatletter
\def\imod#1{\allowbreak\mkern10mu({\operator@font mod}\,\,#1)}
\makeatother

\newcommand{\cT}{\mathcal{T}}

\newcommand{\defn}[1]{{\bf{#1}}}

\newcommand{\M}{{\mathcal{M}}}

\newcommand{\T}{{\mathcal{T}}}

\renewcommand{\L}{\\L}

\newcommand{\Naturals}{\mathbb{N}}
\newcommand{\Nats}{\Naturals}

\DeclareMathOperator{\id}{id}

\renewcommand{\Pr}{\mathbf{P}}
\newcommand{\defas}{:=}

\newcommand{\st}{\,:\,}

\newcommand{\Ind}[1]{{\mathbf 1}_{#1}}

\def\[#1\]{\begin{align}#1\end{align}}

\newcommand{\converges}{\!\!\downarrow}
\newcommand{\diverges}{\!\!\uparrow}

\newcommand{\tind}{t_{\mathrm{ind}}}

\newcommand{\dcut}{d_{\square}}
\newcommand{\delcuthat}{\widehat{\delta}_{\square}}
\newcommand{\delcut}{\delta_{\square}}
\newcommand{\dprok}{d_{w}}

\newcommand{\normcut}[1]{\Vert{#1}\Vert_\square}
\newcommand{\normcutbig}[1]{\bigl\Vert{#1}\bigr\Vert_\square}

\newcommand{\normone}[1]{\Vert{#1}\Vert_1}

\newcommand{\Reals}{\ensuremath{\mathbb{R}}}

\newcommand{\cK}{\mathcal{K}}
\newcommand{\cP}{\mathcal{P}}

\newcommand{\UPn}{U_{\cP_n}}

\newcommand{\zons}{\{0,1\}^{\Nats^2}}

\newcommand{\seqn}[1]{\<{#1}\>_{n\in\Nats}}
\newcommand{\seqi}[1]{\<{#1}\>_{i\in\Nats}}
\newcommand{\seqj}[1]{\<{#1}\>_{j\in\Nats}}
\newcommand{\seqk}[1]{\<{#1}\>_{k\in\Nats}}

\newcommand{\seqiltj}[1]{\<{#1}\>_{i<j\in\Nats}}

\definecolor{RedColor}{rgb}{0.7,0.1,.1}

\definecolor{darkred}{rgb}{0.5,0,0}
\definecolor{darkgreen}{rgb}{0, 0.3,0}
\definecolor{darkblue}{rgb}{0,0,0.6}
\definecolor{LightGray}{rgb}{.6,.6,.6}

\usepackage{inconsolata}

\begin{document}

\title[On the computability of graphons]
{On the computability of graphons}

\author[Ackerman]{Nathanael~L.~Ackerman}
\address{
Harvard University\\
Cambridge, MA 02138\\
USA}
\email{nate@math.harvard.edu}

\author[Avigad]{Jeremy Avigad}
\address{
Carnegie Mellon University \\
Pittsburgh, PA 15213\\
USA
}
\email{avigad@cmu.edu}

\author[Freer]{Cameron~E.~Freer}
\address{
	Borelian Corp.\\ 
	Cambridge, MA 02139\\
	and Remine\\
	Fairfax, VA 22031\\
	USA
	}
\email{freer@borelian.com}

\author[Roy]{Daniel~M.~Roy}
\address{
	University of Toronto\\
	Toronto, ON M5S 3G3\\
	Canada
}
\email{droy@utstat.toronto.edu}

\author[Rute]{Jason~M.~Rute}
\address{
Penn State University\\
University Park, State College, PA 16802\\
USA
}
\email{jmr71@math.psu.edu}


\begin{abstract}
	We investigate the relative computability of exchangeable binary relational data when presented in terms of the distribution of an invariant measure on graphs, or as 
	a graphon in either $L^1$ or the cut distance.
	We establish basic computable equivalences, and show that $L^1$ representations contain fundamentally more computable information than the other representations, but that $0'$ suffices to move between computable such representations. We show that $0'$ is necessary in general, but that in the case of random-free graphons, no oracle is necessary. We also provide an example of an $L^1$-computable random-free graphon that is not weakly isomorphic to any graphon with an a.e.\ continuous version.
\end{abstract}


{\let\thefootnote\relax}


\maketitle

\setcounter{page}{1}
\thispagestyle{empty}

\begin{small}
\renewcommand\contentsname{\!\!\!\!}
\setcounter{tocdepth}{2}
\tableofcontents
\end{small}

\mynewpage

\section{Introduction}

A sequence of random variables is exchangeable when its distribution does not depend on the ordering of its elements.
A well-known theorem of de Finetti shows that infinite exchangeable sequences of random variables are conditionally independent and identically distributed (i.i.d.), 
meaning that, by introducing a new random variable and conditioning on its value,
any apparent dependence between the random variables in the exchangeable sequence is removed and all the variables have the same distribution.

Exchangeable sequences are models for homogeneous data sets 
and serve as building blocks for statistical models with more interesting dependency structures. 
Conditional independence and exchangeability are also central to the probabilistic programming.
Infinite exchangeable sequences arise naturally in functional probabilistic programming languages. 
Indeed, any sequence of evaluations of a closure is a finite prefix of an infinite exchangeable sequence.
The sequence is even manifestly conditionally i.i.d.:
conditioning on the closure itself, every evaluation is independent and identically distributed.

The more interesting phenomenon is the existence of (potentially stateful) probabilistic code
with the property that  
repeated evaluations produce an exchangeable sequence yet no existing variable renders the sequence conditionally i.i.d.
Exchangeability, nevertheless, licenses a programmer or compiler to commute and even prune these repeated evaluations.
These types of transformations are central to several probabilistic programming systems, 
including Church \cite{GMRBT08} and Venture \cite{Venture}.

As described in \cite{PPS2017}, a fundamental question for probabilistic programming 
is whether or not support for exchangeability is in some sense necessary on the grounds of efficiency or even computability.
By de Finetti's theorem, an infinite exchangeable sequence of random variables $X_1,X_2,\dots$ admits a representation $X_j = f(\theta,\xi_j)$ a.s., 
for all $i \in \Nats$, where $f : \Omega \times [0,1] \to [0,1]$ is a measurable function and $\theta,\xi_1,\xi_2,\dots$ are random variables, 
with $\xi_1,\xi_2,\dots$ i.i.d. and independent also from $\theta$.
Clearly, there are many such $f$. On the other hand, any such $f$ could be taken to be a concrete ``representation'' of the conditional independence underlying the sequence $X_1,X_2,\dots$,
and so it is natural to ask when some such $f$ is computable. 
(Indeed, it suffices to establish the computability of the distribution of random probability measure $\nu = \Pr[f(\theta,\xi_1) \in \cdot | \theta]$.)
This question was studied in the setting of exchangeable sequences of real random variables in \cite{MR2880271}, which established the computability of the distribution of $\nu$, and showed that it was even uniformly computable from the distribution of the exchangeable sequence,
yielding an effectivization of de~Finetti's theorem that acts like a program transformation.

Here we study a generalization of de Finetti's theorem to two dimensions,
and in particular,
the binary symmetric case of the Aldous--Hoover theorem for two-dimensional jointly exchangeable arrays.
Our focus will be on the computability of various representations of the distribution of these two-dimensional arrays, 
with a special emphasis on the representation of graphons. (The function $f(\theta,\cdot)$ is a one-dimensional analogue of a graphon.
For more details on graphons and exchangeable arrays, see \cite{MR3012035} and \cite{PSIP}, respectively.) 
The situation here is more complicated, and depends on the choice of metric. 
Several of the standard metrics are computably equivalent to each other, 
but we find that one natural way of expressing the relevant measurable object --- corresponding to the so-called edit distance --- is not computable from the distribution of the exchangeable data itself, 
unlike in the one-dimensional case. 
This mapping is possible using the halting problem $0'$ as an oracle, 
and we provide an example showing that this dependence is necessary. 
On the other hand, for a natural subclass, the so-called random-free case, 
one can recover the graphon in the edit distance metric from the distribution of the array.

In the case of computable distributions on binary symmetric exchangeable arrays, 
all these results fall short of establishing that some graphon is computable on a measure one set of points.
In fact, one cannot hope for such representations. 
We show that, in the case of computable distributions on arrays, 
there need not be any graphon generating the array
that is almost-everywhere continuous, 
providing a fundamental topological impediment to computable representations that would render it continuous on a measure-one set.

It follows that the two-dimensional setting is fundamentally different from the one-dimensional one. 
There need not be a representation of an exchangeable array that exposes the conditional independence inherent in the exchangeable array,
even if the exchangeable array itself has some computable representation.
In some special cases, we can compute an $L^1$ representation,
which exposes the conditional independence, but only allows us to represent the distribution of the exchangeable array up to an arbitrarily small, but nonzero, error term.
Our results suggest that probabilistic programming languages may need to have special support for probabilistic symmetries, such as exchangeability in the case of arrays.

\subsection{Summary of main results}

In this paper, we examine four main representations of invariant measures on graphs, as elements of complete pseudometric spaces, as we describe in
Section~\ref{background-sec}.
We consider the space of invariant measures with a metric equivalent to weak convergence, and the space of graphons under the metrics $d_1$, $\dcut$, and $\delcut$. 
The distance $d_1$ is the $L^1$ pseudometric on graphons, and is closely related to edit distance on graphs. The $\delcut$ distance between two graphons can be seen as measuring the difference in distribution between graphs sampled from them, and is obtained from the simpler $\dcut$ upon taking the infimum under measure-preserving transformations of the graphon.

We establish the key computable relationships between these representations in Section~\ref{comprep-sec}, using the framework of computable analysis. In particular, we show that there is a computable equivalence between names for invariant measures and $\delcut$-names for the corresponding graphons, and show that the identity map from $\dcut$-names to $\delcut$-names has a computable section.

While every $d_1$-name is also a $\dcut$-name, in Section~\ref{upper-bound} we show how to transform a computable $\dcut$-name into a $d_1$-name using the halting problem $0'$.
In Section~\ref{lower-bound} we show that $0'$ is necessary, by constructing a graphon with a computable $\dcut$-name such that $0'$ can be computed from any $d_1$-name --- although $d_1$ names are computable from $\dcut$ names in the random-free case.
Finally, we establish in Section~\ref{ae-cont-sec} that there is a $d_1$-computable graphon that is not even a.e.\ continuous in any topology consisting of Borel sets.

\subsection{Notation}
	Let $\lambda$ denote Lebesgue measure on $\Reals$ (though for notational convenience, we will often use $\lambda$ to also refer to Lebesgue measure on $\Reals^2$, etc.).
	For $n\in\Nats$, write $[n] \defas \{0, \ldots, n-1\}$.
	All logarithms, written $\log$, will be in base 2.

	For $e\in\Nats$, let $\{e\}$ denote the partial computable function $\Nats \to \Nats$ given by computer program $e$. For $n\in\Nats$ write $\{e\}(n)\converges$ 
	to denote that this program halts 
	on input $n$,  and 
	$\{e\}(n)\diverges$ to denote that it
	does not halt.
	For $s\in\Nats$ write $\{e\}_s(n)\converges$ 
	to denote that this program has halted 
	after at most $s$ steps,
	and $\{e\}(n)_s\diverges$ to denote that it
	has not yet halted after $s$ steps.

Write $0' \defas \{ e\in\Nats \st \{e \}(0)\converges\}$ to denote the halting set.
	Recall that a set $X\subseteq \Nats$ is (co-)c.e.\ complete when it (its complement) is computably enumerable and any computably enumerable set admits a computable many-one reduction, called an \emph{$m$-reduction}, to it (its complement, respectively). For example, $0'$ is c.e.\ complete.

\section{Graphons and invariant measures on graphs}
\label{background-sec}

In this section we provide the basic definitions and results about the two main objects of interest in this paper, graphons and invariant measures on graphs.

\subsection{Graphons}
\label{graphon-subsec}

We will formulate most of our results in terms of graphons, both for concreteness and simplicity of notation. 
In this subsection we summarize the standard notions that we will need.
		For more details on graphons, their basic properties, and notation, we refer the reader to \cite[Chapter~7]{MR3012035}, from which most of the definitions in this subsection are borrowed.
In many cases, analogous notions and results were developed earlier in terms of exchangeable arrays; for details on this connection and the history, see \cite{MR2426176} and \cite{MR2463439}.

	\begin{definition}
		A \defn{graphon} is a symmetric measurable function $W\colon [0,1]^2 \to [0,1]$. 
Let $\Graphons$ denote the set of all graphons.
	\end{definition}

	Three classes of graphons play a special role in this paper.

	\begin{definition}
		A graphon $G$ is \defn{random-free} if it is
		$\{0, 1\}$-valued a.e., in other words, when $\lambda\bigl(G^{-1}(\{0, 1\})\bigr) = 1$.
	\end{definition}

	\begin{definition}
		Let $\T$ be a topology on $[0, 1]$. A graphon $G$ is \defn{almost everywhere (a.e.)\ continuous} with respect to $\T$ if there is a set $X \subseteq [0,1]^2$ such that $\lambda(X) = 1$ and $G|_{X}$ is continuous with respect to $\T|_X$. 
	\end{definition}

	\begin{definition}
		A graphon $G$ is a \defn{step function} if there is a (finite) measurable partition $P$ of $[0,1]$ such that for every pair of parts $p, q \in P$, the graphon
		$G$ is constant on $p \times q$.  
		Write $\StepGraphons$
		to denote the class of step functions whose underlying partition divides $[0,1]$ into some finite number of equally-sized intervals, and whose range is contained in the rationals.
		Write $\StepGraphonsRF$ to denote the subset of its random-free graphons.
		Each of $\StepGraphons$ and $\StepGraphonsRF$ admits a straightforward computable enumeration.
	\end{definition}

	One can associate to each finite graph (on $[n]$ for some $n\in\Nats$) a random-free step function graphon, 
	as we now describe, 
	such that the space of graphons is the completion of the finite graphs (embedded this way) under an appropriate pseudometric on finite graphs.

	\begin{definition} 
		Let $G$ be a graph with vertex set $[n]$, for some $n\in\Nats$. Define the step function graphon $W_G\colon [0,1]^2 \to \{0,1\}$ to be such that $W_G(x,y) = 1$ if and only if there is an edge between $\floor{n x}$ and $\floor{n y}$ in $G$.
		In other words, $W_G$ is equal to $1$ on $[i/n,(i+1)/n) \times [j/n,(j+1)/n)$ if there is an edge from $i$ to $j$, and $0$ otherwise.  Observe that $W_G \in\StepGraphonsRF$.
	\end{definition}

	We will use three key pseudometrics on the space of graphons. We begin by describing the cut norm $\normcut{\cdot}$, which will allow us to define $\dcut$ and then $\delcut$, the latter of which is closely related to subsampling.
	The third pseudometric is $d_1$, more closely related to edit distance and the $L^1$ norm.

	\begin{definition}
		The	\defn{cut norm} of a symmetric measurable function $F\colon [0,1]^2 \to [-1,1]$ is defined by
\[
\normcutbig{F} \defas
\sup_{S,T\subseteq[0,1]}\left|\int_{S\times T}F(x,y)\, dx\, dy\right|
\]
where $S$ and $T$ range over measurable sets. 
		For graphons $U$ and $W$, define $\dcut(U, W)\defas \normcut{U - W}$.
	\end{definition}

	It is straightforward to verify that $\normcut{\cdot}$ is a norm on $\Graphons$, and that $d_\cut$ is a pseudometric on $\Graphons$.

	As we will see, this cut norm is too coarse for many of our purposes. 

	\begin{definition}
		The \defn{$L^1$ norm} of a symmetric measurable function $F\colon [0,1]^2 \to [-1,1]$ is defined by 
		\[
			\textstyle
			\normone{G} \defas \int \left |F(x,y) \right| \, dx\, dy .
			\]
			For graphons $U$ and $W$,
		define 
				$d_1(U, W) \defas \normone{U-W}$.
	\end{definition}

	It is a standard fact that the $L^1$ norm is a norm on $\Graphons$, and that $d_1$ is a pseudometric on $\Graphons$. 

	While the cut norm is much coarser than the $L^1$ norm, they do agree on the notion of norm zero.
The following easy lemma follows from the fact that if $\normcut{W} = 0$ then $W = 0$ a.e. 

	\begin{lemma}
		\label{zerolemma}
		If $W$ is a graphon, then $\normone{W} = 0$ if and only if $\normcut{W} = 0$. 
	\end{lemma}

	This lemma implies that 
	the pseudometrics $d_1$ and $d_\cut$ can be thought of as metrics on the same quotient space, namely $\Graphons / \{(G, H)\st d_1(G, H) = 0\}$, 
	even though the metrics they induce on this space are very different.

	As we will see in \S\ref{inv-meas-subsec},
	there is a standard way to associate to each graphon an invariant measure on countable graphs
	and given two graphons we would like to have a condition equivalent to the corresponding invariant measures being the same. However, it is easy to see, by applying a non-trivial measure-preserving transformation to any non-constant random-free graphon,
	that there are graphons which give rise to the same distribution but which are very far in either $d_1$ or $\dcut$.
	Hence we will need an even coarser notion of distance, which we now define. 

	\begin{definition}
		Let $W$ be a graphon and let $\varphi\colon ([0,1], \lambda) \to ([0,1], \lambda)$ be a measure-preserving map.
		Define $W^\varphi$ to be the graphon satisfying
		\[
			W^\varphi(x,y) = W\bigl(\varphi(x), \varphi(y)\bigr)
			\]
			for all $x,y\in[0,1]$.
	\end{definition}

	\begin{definition}
		\label{delcut-def}
		For graphons $U$ and $W$, define
		\[
			\delcut(U, W) \defas \inf_{\varphi} d_\cut(U, W^\varphi),
			\]
where the infimum is taken over all measure-preserving maps of $([0,1], \lambda)$ to itself. 
	\end{definition}

	The following standard result will be important when we consider the computability of the
	representations of graphons in these various metric spaces.

	\begin{lemma}
		The set $\StepGraphons$ 
		is dense in $(\Graphons, d_1)$, and its subset $\StepGraphonsRF$ 
		is dense in $(\Graphons, \dcut)$ and $(\Graphons, \delcut)$. 
	\end{lemma}
	\begin{proof}
		The density of step functions in $d_1$ is a standard measure-theoretic fact.
		The density of $\StepGraphonsRF$ in $\delcut$ follows from
\cite[Theorem~11.52]{MR3012035}.
This implies that 
		\[
			\bigl\{ W^\varphi \st W \in \StepGraphonsRF \textrm{~and~} \varphi \textrm{~is~a~measure-preserving~map}\bigr\}
			\]
			is dense in $\dcut$.
			But 
		for every $W\in \StepGraphonsRF$, measure-preserving map $\varphi$, and 
		$\epsilon> 0$, there is an element $V \in \StepGraphonsRF$ such that
		\[
			\dcut(W^\varphi, V) < \epsilon.
			\]
		Hence $\StepGraphonsRF$ is also dense in $\dcut$.
	\end{proof}

We will later need the following definition.
	\begin{definition} A graphon $W$ is \defn{twin-free} if for each pair of distinct points $x, y\in[0,1]$, the functions $z\mapsto W(x,z)$ and $z\mapsto W(y,z)$ disagree on a set of positive Lebesgue measure. 
	\end{definition}

\subsection{Invariant measures on graphs}
\label{inv-meas-subsec}

Invariant measures on graphs with underlying set $\Nats$
are the main object of study in this paper.
In the probability theory literature, one often studies exchangeable arrays rather than their distributions, but here we focus on their distribution as we will be interested in the measures rather than the random variables, and so that we can avoid certain technicalities and notational difficulties.

\begin{definition}
	Let $\Graphs\subseteq \zons$ denote the space of adjacency matrices of symmetric irreflexive graphs with underlying set $\Nats$.
	A probability measure $\mu$ on the space $\Graphs$ is called an \defn{invariant measure on graphs} if 
		$\mu(A) = \mu(\sigma^{-1}(A))$
	for all Borel $A \subseteq \Graphs$ and all permutations $\sigma\colon \Nats \to \Nats$.
\end{definition}

We will use the term \emph{invariant measure} to refer to invariant measures on graphs.

An important subclass of the invariant measures are those that are extreme. 

\begin{definition}
	An invariant measure $\mu$ is 
	\defn{extreme} 
	if there do not exist 
	invariant measures $\nu$ and $\pi$
	such that $\mu = \alpha \nu + (1-\alpha) \pi$ for some $\alpha \in (0, 1)$. 
\end{definition}
In our setting, the extreme measures coincide with the ergodic ones (with respect to permutations of $\Nats$).

Graphons naturally give rise to extreme 
invariant measures on graphs,
via the distribution of the countably infinite random graph obtained by sampling from the graphon, as we now describe. 

\begin{definition}
Let $W$ be a graphon and let $S$ be a countable set. 
	Let $\<\zeta_i\>_{i\in S}$ be an i.i.d.\ collection of uniform $[0,1]$-valued random variables.
	Consider the random graph 
		$\widehat{\G}(S, W)$ with vertex set $S$ where for all distinct $i,j\in S$, there is an edge between $i$ and $j$ independently, with probability $W(\zeta_i, \zeta_j)$.
	For $n\in\Nats$, we write $\widehat{\G}(n, W)$ to refer to $\widehat{\G}([n], W)$.
	When $H$ is a finite graph, we write $\widehat{\G}(S, H)$ to refer to $\widehat{\G}(S, W_H)$. 
		Finally, let $\G(S,W)$ denote the distribution of $\widehat{\G}(S,W)$.
\end{definition}
For $0<p<1$, if $W$ is the constant graphon $W \equiv p$, then $\G(\Nats, W)$ is the distribution of an \ER\ random graph.

The following lemma is standard. 
\begin{proposition}[{\cite[Theorem~11.52]{MR3012035}}]
	If $W$ is a graphon, then $\G(\Nats, W)$ is an extreme invariant measure on graphs.
\end{proposition}

Conversely, every extreme invariant measure arises from a graphon.

\begin{proposition}[{\cite[Theorem~11.52]{MR3012035}}]
	If $\mu$ is an extreme invariant measure on graphs then there is some graphon $W$ such that 
$\G(\Nats, W)$ and $\mu$ are the same distribution.
\end{proposition}

It is then natural to ask when
two graphons give rise to the same invariant measure. 

\begin{theorem}[{\cite[Theorem~13.10]{MR3012035}}]
	\label{weakly-isomorphic-defthm}
	For graphons $U$ and $W$, the following are equivalent.
	\begin{enumerate}
		\item $\G(\Nats, U)$ and $\G(\Nats, W)$ are the same distribution.

		\item $\delcut(U, W) = 0$.

		\item There are measure-preserving maps $\varphi, \psi\colon [0,1] \to [0,1]$ such that $U^\varphi = W^\psi$ a.e.
	\end{enumerate}
	When any of these equivalent conditions holds, we say that $U$ and $W$ are \defn{weakly isomorphic}.
\end{theorem}

We now describe a natural metric on the space of invariant measures.

\begin{definition}
	Let $\ExchMeasures$ be the collection of extreme invariant measures, let $\mu \in \ExchMeasures$, and let $F$ be a finite graph on $[n]$. Define $\tind(F, \mu) \defas \mu( \{ G\in \Graphs \st G|_{[n]} = F\})$. 
Fix an enumeration $\<F_i\>_{i \in \Nats}$ of finite graphs with underlying set $[n]$ for some $n \in \Nats$. For $\mu, \nu \in \ExchMeasures$, define 
\[
\dprok(\mu, \nu) \defas \sum_{i \in \Nats} 2^{-i} \, \bigl|\tind(F_i, \mu) - \tind(F_i, \nu)\bigr|.
\]
\end{definition}

The following is standard. 

\begin{lemma} The space 
	$(\ExchMeasures, \dprok)$ of extreme invariant measures is a compact Polish space with the topology of weak convergence. Further, \[
		\SampledFiniteGraphs \defas \bigl\{\G(\Nats, G) \st G \text{~is a finite graph} \bigr\}
	\]
	is a dense subset. 
\end{lemma}
Note that $\SampledFiniteGraphs$ also admits a straightforward computable enumeration.

The previous lemma tells us that we can approximate an extreme invariant measure arbitrarily well by measures which come from sampling graphons induced by finite graphs. A natural question is whether it is possible to take an invariant measure and find a (possibly random) sequence of finite graphs whose corresponding graphons almost surely converge to the invariant measure we started with. 
This is possible, as the following result states.

\begin{lemma}[{\cite[Corollary~11.15]{MR3012035}}]
	\label{weak-as-convergence-of-random-graphs}
	Suppose $U$ is a graphon. Then 
	\[
		 \bigl\<\G\bigl(\Nats, \widehat{\G}(n,U)\bigr)\bigr\>_{n\in\Nats}
		\]
		is a random sequence of extreme invariant measures that almost surely converges in $(\ExchMeasures, \dprok)$ to $\G(\Nats, U)$.
\end{lemma}

\section{Notions of computability for graphons 
and invariant measures on graphs}
\label{comprep-sec}

We begin by describing the notions of computability for graphons and invariant measures on graphs, and then present some of the basic relationships between them.

\subsection{Computable pseudometric spaces}

In order to describe computable elements of the spaces of graphons and invariant measures with respect      to various pseudometrics, we will use the notion of a computable pseudometric space,
a straightforward generalization of the notion of a computable metric space in computable analysis (see, e.g., \cite{MR2519075}).

\begin{definition}
	A \defn{computable} (complete) pseudometric space consists of a triple $(M, d, \seqi{s_i})$ such that 
	\begin{itemize}
		\item $(M, d)$ is a complete pseudometric space, 

		\item $\seqi{s_i}$ is dense in $(M, d)$, and 

		\item the sequence $\seqiltj{d(s_i, s_j)}$ is a computable sequence of real numbers. 
	\end{itemize} 
\end{definition}

\begin{definition}
	Suppose $(M, d, \seqi{s_i})$ is a computable pseudometric space. A \defn{rapidly converging Cauchy sequence} is a sequence $\seqj{s_{k_j}}$ for which  
	\[
		d(s_{k_j}, s_{k_\ell}) \leq 2^{-j}
		\]
		for $j < \ell \in \Nats$.
 
		A rapidly converging Cauchy sequence is called a \defn{$d$-name} for the limiting value $\lim_{n \to \infty} s_{k_n}$. We say that $\seqj{s_{k_j}}$ is \defn{computable}  in $d$
		if the sequence of natural numbers $\seqj{k_j}$ is computable, and that an element $s \in M$ is \defn{computable} if it has some $d$-name that is computable. (These notions relativized to an oracle are defined in the obvious way.)
\end{definition}

Roughly, a name for an element of the pseudometric space is a sequence of approximations that converges with rate $n \mapsto 2^{-n}$. Note that the choice of this rate is somewhat arbitrary, since given a sequence that converges with some other computable rate, one can computably ``thin out'' the sequence so that it converges at the rate we have chosen.

The computational strength needed to produce a $d$-name provides a measure of the complexity of the corresponding element of the represented space.

The pseudometric spaces we have considered so far can be straightforwardly made into computable pseudometric spaces using the computable enumerations of dense subsets we have identified.

\begin{lemma}
	The following are computable pseudometric spaces:
	\begin{itemize}
		\item $(\Graphons, d_1, \StepGraphons)$,

		\item $(\Graphons, \dcut, \StepGraphons)$,

		\item $(\Graphons, \delcut, \StepGraphons)$, and

		\item $(\ExchMeasures, \dprok, \SampledFiniteGraphs)$. 
	\end{itemize}
\end{lemma}

In this paper we are interested in the relative computability of names for graphons and invariant measures considered as elements in these various computable pseudometric spaces.

\subsection{Computable relationships between representations}

In this section we want to consider the computable relationship between various representations of graphons and exchangeable arrays. In order to do this we need a notion of a \emph{computable function} two pseudometric spaces. 

\begin{definition}
\label{definition:computable:function}
	Suppose $(M, d, S)$ and $(N, f, T)$ are computable pseudometric spaces. We say a map $g\colon M \to N$ is a \defn{computable function}, or is simply \defn{computable}, if there is a computer program $e$ such that whenever $K \defas \<k_j\>_{j \in \Nats}$ is an index sequence for a $d$-name of an element $a$ then $\{e\}^{K, \ell}$ outputs an index set for an $f$-name of $g(a)$. 

	Suppose $h \colon N\to M$ is a computable map.
	We say that a computable map $g\colon M \to  N$ is a \defn{computable equivalence witnessed by $h$} if $d\bigl(x, h(g(x))\bigr) = 0$ for all $x\in M$
	and $f\bigl(y, g(h(y))\bigr) = 0$ for all $y\in N$.
	In this case we say that the spaces are \defn{computably equivalent}.

	Let $k \colon M \to N$ be a surjective function.
	A computable map $h \colon N\to M$ is a \defn{computable section} for $k$ if
	$f\bigl(y, k(h(y))\bigr) = 0$ for all $y\in N$.
\end{definition}

In other words, a function is computable if there is an algorithm that takes a name in one space and computably transforms it into a name in the other. A computable equivalence provides a uniform method for transforming a name in one space to a name in the other and vice-versa. Note that a computable equivalence induces a bijection between the corresponding metric spaces obtained by taking the quotient by distance $0$ on each side.

We will consider computable sections in the case where the underlying sets $M$ and $N$ are the same and $k$ is the identity function. In this case, a computable section takes an $f$-name for a computable element of $N$ and returns a $d$-name for a (possibly different) computable element of $M$ such that $(N, f, T)$ cannot ``distinguish'' the points, in the sense that they have $f$-distance $0$.

Consider the following notions for an invariant measure $\mu$.
\begin{itemize}
	\item[(1)] There is a computable $\dprok$-name for $\mu$.
	
	\item[(2)] There is a graphon $W$ with a computable $\delcut$-name such that $\G(\Nats, W) = \mu$.

	\item[(3)] There is a graphon $W$ with a computable $\dcut$-name such that $\G(\Nats, W) = \mu$.

	\item[(4)] There is a graphon $W$ with a computable $d_{1}$-name such that $\G(\Nats, W) = \mu$.
\end{itemize}

The next theorem establishes relationships between these four notions
		which yield the implications in Corollary~\ref{implications}.
In fact, as we will later see, these implications are all that are possible.

\begin{theorem}
	\label{basic-relationships}
The following functions between pseudometric spaces are computable.
\begin{itemize}
	\item[(a)] A map $\alpha\colon  (\ExchMeasures, \dprok) \to (\Graphons, \delcut)$ which takes $\G(\Nats,W)$ to some graphon $U$ weakly isomorphic to $W$.

\item[(b)] The map $\beta\colon (\Graphons, \delcut) \to (\ExchMeasures, \dprok)$ which takes $W$ to $\G(\Nats,W)$.

\item[(c)] The identity map $\id\colon (\Graphons, \dcut) \to (\Graphons, \delcut)$.

\item[(d)] The identity map $\id\colon (\Graphons, d_1) \to (\Graphons, \dcut)$.
\end{itemize}
Furthermore, $\alpha$ is a computable equivalence witnessed by $\beta$, and vice-versa, and 
	there is a computable section of the identity map $\mathrm{(c)}$.
	Finally, $\mathrm{(a)}$, $\mathrm{(b)}$, and $\mathrm{(d)}$ induce bijections on the corresponding metric spaces. 
\end{theorem}

\begin{proof}
	The $d_1$-distance between two graphons is at least their $\dcut$-distance, and so
	any $d_1$-name is a $\dcut$-name. Similarly,
	any $\dcut$-name is a $\delcut$-name. 
	Hence (c) and (d) are computable.

	Now to show (b) we want to show that if $\seqn{W_{G_n}} \subseteq \StepGraphons$ is a rapidly convergent Cauchy sequence in $\StepGraphons$ then $\seqn{\G(\Nats, W_{G_n})}$ is a rapidly convergent sequence in $\ExchMeasures$.  
	By the Counting Lemma (\cite[Exerise~10.30]{MR3012035}),
	for any graphons $U$ and $V$ and finite graph $F$ with $k$ vertices,
	we have
	\[
		\textstyle
		\bigl | \tind(F,U)  - \tind(F,V) \bigr| \le 4 {k \choose 2} \dcut(U,V).
		\]
	For any measure-preserving map $\varphi\colon [0,1] \to [0,1]$, we have $\tind(F, V) = \tind(F,V^\varphi)$.  Therefore 
	\[
		\textstyle
		\bigl | \tind(F,U)  - \tind(F,V) \bigr| \le 4 {k \choose 2} \delcut(U,V)
		\]
		holds
	by Definition~\ref{delcut-def}. 
	Hence given a $\delcut$-name for $U$ we can computably thin out its entries to form a $\dprok$-name.

To show (a) we need the following Chernoff bound to obtain a representation in terms of $\delcut$ given the distribution of an exchangeable array.
As shown in \cite[Lemma~10.16]{MR3012035},
	for each $k\in\Nats$,
with probability at least $1-e^{-k/(2\log k)}$, we have
\[
	\delcut\bigl(U, W_{\widehat{\G}(k,U)} \bigr)\leq\frac{22}{\sqrt{\log k}}.
\]
	As $\G(k, U)$ depends only on the distribution of the induced exchangeable array we can find an element $G_k$ which is within $\frac{44}{\sqrt{\log k}}$ in $\delcut$ of 
	$W_{\widehat{\G}(k,U)}$ with probability 
	at least $1-e^{-k/(2\log k)}$, and so in particular 
	\[
		\delcut(G_k, U) \le\frac{44}{\sqrt{\log k}}.
		\]
	This therefore lets us create a sequence $\seqk{G_n}$ such that $\seqk{W_{G_k}}$ is a $\delcut$-name for any graphon with the same distribution as $U$. 

	To establish the computable equivalences in (a) and (b), we use the fact that convergence in $\delcut$ is equivalent to convergence of the corresponding random graphs (Theorem~\ref{weakly-isomorphic-defthm}).
	This also shows that (a) and (b) induce bijections on the corresponding metric spaces.

We now show that (c) has a computable section.
	Assume we have a $\delta_{\square}$-name of $U$,
	and
	know a graph $G_{n}$ such that 
	\[
		\delta_{\square}(W_{G_{n}},U)<2^{-(2^{2n}+1)}.
		\]
We will find a graph $G_{n+1}$ such that 
	\[
		\dcut(W_{G_{n+1}},W_{G_{n}})\leq45 \cdot 2^{-n}
		\]
	and 
	\[
		\delta_{\square}(W_{G_{n+1}},U)\leq2^{-({2^{2(n+1)}+1)}}.
		\]
		This is enough
to get a fast Cauchy sequence in $\dcut$.
	Find $H$ such that 
	\[
		\delta_{\square}(W_{H},U)<2^{-({2^{2(n+1)}+1)}}.
		\]
		Then
	\[
	\delta_{\square}(W_{H},W_{G_{n}})<2^{-(2^{2n}+1)} + 2^{-(2^{2(n+1)}+1)} < 2^{-2^{2n}}.
	\]
	There are graphs $G'_n$ and $H'$ both on the set $[|G_n|\cdot |H|]$ (where $|G_n|$ denotes the number of vertices of $G_n$, and similarly with $H$) such that 
	\[
		\dcut(W_{G'_n}, W_{G_n}) = \dcut(W_{H'}, W_{H}) = 0,
		\]
		by taking blow-ups to a common refinement.
	Following the notation in \cite[\S8.1.3]{MR3012035}, define the quantity
	\[
		\delcuthat(G'_n, H') \defas \min_{\widehat{G'}_n, \widehat{H'}} \dcut(W_{\widehat{G'}_n}, W_{\widehat{H'}}),
		\]
	where
	$\widehat{G'}_n$ ranges over the images of $G'_n$ under permutations of $[|G_n|\cdot |H|]$, and similarly with $\widehat{H'}$. By  \cite[Theorem~9.29]{MR3012035}, we have
	\[
		\delcuthat(G'_n, H')
		\le \frac{45}{\sqrt{- \log \delcut(W_{G'_n}, W_{H'})}}
		< \frac{45}{\sqrt{2^{2n}}}
		= 45\cdot 2^{-n}
		.
		\]
	Hence there is some reordering $G_{n+1}$ of $H'$ such that 
	\[
		\dcut(W_{G_{n+1}}, W_{G'_n}) \le 45\cdot 2^{-n}.
		\]
		Because $\dcut(W_{G'_n}, W_{G_n}) = 0$, we have 
	\[
		\dcut(W_{G_{n+1}}, W_{G_{n}}) \le 45\cdot 2^{-n}.
		\]
	By definition, $\delcut(W_{G_{n+1}}, W_{H'}) = 0$.
	Because 
	$\dcut(W_{H'}, W_{H}) = 0$, we therefore have
	\[
	\delcut(W_{G_{n+1}}, U) =\delcut(W_{G_{n+1}}, H) < 2^{-({2^{2(n+1)}+1)}},
	\]
	as desired.

	Finally, (d) induces a bijection on the corresponding metric spaces, as noted in Lemma~\ref{zerolemma}.
\end{proof}

	We have seen there there is a computable equivalence between $\delcut$-names for a graphon and names for the corresponding invariant measure.
Further, given a $\delcut$-name,
we can computably find a $\dcut$-name for a graphon yielding the same invariant measure.
	We have also seen that it is possible to transform a $d_1$-name to a $\dcut$-name in a computable way. It is therefore natural to ask whether there is a computable equivalence from a $\dcut$-name to a $d_1$-name. 
As we will see, in general there is not. This tells us that the $d_1$-name for a graphon contains fundamentally more computable information than an $\dcut$-name for a graphon. 

As a consequence of
	Theorem~\ref{basic-relationships}, we obtain the following relationships among the numbered notions appearing after Definition~\ref{definition:computable:function}.

	\begin{corollary}
		\label{implications}
		For an invariant measure $\mu$, notions $\mathrm{(1)}$, $\mathrm{(2)}$ and $\mathrm{(3)}$ are equivalent, and are all implied by notion $\mathrm{(4)}$.
	\end{corollary}

	Our later results show that (4) is not implied by (1), (2), or (3).

\section{$\dcut$-names vs.\ $d_1$-names: Upper bound}
\label{upper-bound}

Recall that a $d_1$-name is already a $\dcut$-name.
In this section we establish that the halting problem $0'$ suffices as an oracle
to computably transform a computable $\dcut$-name to a $d_1$-name.
Further, in the random-free case, this oracle is not needed.
In the next section we show that this is tight, in the sense that the use of $0'$ is necessary in general.

For $k\in\Nats$, let
$\mathcal{P}_k$ denote the equipartition of $[0,1]$ into $2^k$-many intervals of width $2^{-k}$. 
For a graphon $U$ we write
$\UPn$ 
(as in  \cite[\S7.1]{MR3012035})
to denote the step function graphon 
	$\mathbb{E}[U \mid \mathcal{P}_{n} \times \mathcal{P}_{n}]$, i.e., the conditional expectation of the function $U$ averaged on this $2^n \times 2^n$ square grid. 

The following version of the weak regularity lemma for graphons follows immediately from \cite[Lemma~9.15 (a) and (b)]{MR3012035} (taking $m=1$ and $k=\lceil n/4\rceil$ in (a) and then $m = 2^{\lceil n/4\rceil}$ and $k = 2^n$ in (b), analogously to the proof of \cite[Lemma~10.16]{MR3012035}).

\begin{lemma}[{\cite[Lemma~9.15]{MR3012035}}]
	\label{weakregularity}
Let $U$ be a graphon. Then for all $n$,
	$\dcut(\UPn,U) \leq 8/\sqrt{n}$.
\end{lemma}

Consider the space $\mathcal{M}$ of all martingales $\seqn{f_n}$
where $f_n$ is a $[0,1]$-valued step function graphon that is $\mathcal{P}_n \times \mathcal{P}_n$-measurable.
Recall that being a martingale means that for every $n$, we have $\mathbb{E}[f_{n+1} \mid \mathcal{P}_n \times \mathcal{P}_n] = f_n$.

We endow this space with a natural topology which makes it effectively compact.  In particular, we can view $\mathcal{M}$ as a closed subspace of the compact space $[0,1]^\mathbb{N}$ as follows.  
For each martingale
$f=\<f_n\>_{n\in\Nats}$
we define a corresponding element
$x^f= \<x^f_n\>_{n\in\Nats}\in [0,1]^\mathbb{N}$.
Define the first coordinate $x^f_0$ to be the value of the 
constant function $f_0$.
Then define the next 4 coordinates $x^f_1, \ldots, x^f_4$ to be the values of the
$2 \times 2$ step function $f_1$.
Continue this way for each $n$, reading off $2^{n+1}$-many coordinates
from the $2^n \times 2^n$ step function $f_n$. Clearly this is an injection from $\mathcal{M}$ into $[0,1]^\Nats$.
A martingale $f \in \mathcal{M}$ is said to be \defn{computable} if the corresponding sequence $x^f \in [0,1]^\mathbb{N}$ is computable.

We say 
that an element of $[0,1]^\Nats$ 
encodes a martingale when it equals $x^f$ for some martingale $f$.
The subspace of $[0,1]^\Nats$ encoding a martingale is a $\Pi^0_1$ class,
since if $y$ does not code a martingale, then there is a computer program that, given $y$ as an oracle, outputs this fact (by noticing that averaging fails at some level).

Given a graphon $U$ that is computable in $\dcut$, 
we will find a $\Pi^0_1$ subclass of $\M$ that has a single element, which converges in $d_1$ to $U$. To establish the computability of this point, we will use the following lemma.

\begin{lemma}
	\label{singleton}
	Suppose a $\Pi^0_1$ subset of $[0,1]^\Nats$ is a singleton. Then its
	unique member is computable.
\end{lemma}
\begin{proof}
	Let $T = \{\seqn{x_n}\}$ be the singleton set.  We can enumerate
	all the rational cylinder sets $[a_0, b_0] \times \cdots \times [a_{n-1}, b_{n-1}] \times [0,1]^\Nats$ which are disjoint from the set $T$.  We now describe how to compute the coordinate $x_n$ for a given $n\in\Nats$.
	
	Consider any rational numbers $a$ and $b$ such that $a < x_n < b$.  By the compactness of $[0,1]^\Nats$, we will eventually enumerate a finite rational cover of $[0,1]^n \times [0, a] \times [0,1]^\Nats$ and a finite rational cover of $[0,1]^n \times [b, 1] \times [0,1]^\Nats$ since these are compact sets disjoint from $T$.  Moreover, since these enumerated covers are made of rational intervals, we will be able to computably determine when enough has been enumerated to cover the desired sets.  By waiting for such covers to occur, we can computably learn that $a < x_n < b$. 
	
	This argument holds for all rationals $a$ and $b$ for which $a < x_n < b$, and so we can approximate $x_n$ to arbitrary precision.
\end{proof}

We now find the $\Pi^0_1$ subclass of $\M$ that has a single element.

\begin{lemma}
	\label{martingale-comp-from-cut-norm}
	Let $U$ be a graphon computable in $\dcut$. Then the martingale $\seqn{\UPn}$
	is computable.
\end{lemma}
\begin{proof}
	Let $\cK_U \subseteq \M$ be the collection of all martingales $\seqn{f_n}$ 
	satisfying $\dcut(f_{n},U) \leq 8/\sqrt{n}$ for all $n\in\Nats$.
	Because
	$U$ is computable in $\dcut$, 
	$\cK_U$ is a $\Pi^0_1$ class.
Also note that $\cK_U$ is nonempty since the martingale 
$\UPn$
	is in $\cK_U$ 
	by Lemma~\ref{weakregularity}.
	Further, every 
	martingale
	in $\cK_U$ converges in $d_1$
	by the martingale convergence theorem (see, e.g., \cite[Theorem~7.23]{FMP2})
	to a graphon $V$.  
	
	Now pick some $\seqn{f_n} \in \cK_U$.  Since 
	\[
		\dcut(f_n, V) < d_1(f_n, V) \to 0
		\]
		and 
		\[
			\dcut(f_n,U) \leq 8/\sqrt{n} \to 0
			\]
			we have $\dcut(V,U) = 0$
			and hence $U = V$ a.e.  This shows that $\cK_U$ only has one element, namely $\UPn$, and so $\seqn{f_n} = \UPn$.
	Since $\cK_U$ is a $\Pi^0_1$ class and has only one element, that element is computable by Lemma~\ref{singleton}.
\end{proof}

Note that this tells us that $\UPn$ is a computable martingale that converges in $d_1$, but it need not converge quickly. As we will see, in general we cannot computably identify a rapidly converging subsequence.

\begin{theorem}
	Let $U$ be a graphon. Then  from 
	the jump of any $\dcut$-name of $U$, we can compute an $d_1$-name for it.
	In particular, if $U$ has a computable $\dcut$-name, then it has a $0'$-computable $d_1$-name.
\end{theorem}
\begin{proof}
	Let $X$ be some $\dcut$-name for $U$.  
	Then using $X$ we can compute a name for the martingale $\UPn$.  Since $\UPn$ converges in $d_1$, its limit is computable in $d_1$ from the jump of a name for $\UPn$, which in turn is computable from the jump of $X$.
\end{proof}

It is natural to consider the case of random-free graphons, especially since the ability to flip between greyscale regions and black-and-white ones will be key to the lower-bound proof in Section~\ref{lower-bound}.

In fact, in the random-free case, convergence of the martingale is tamer.

\begin{lemma}
	\label{martingale-comp-for-randomfree}
	If $\seqn{f_n}$ is a computable martingale that converges to a random-free graphon $U$, then $U$ has a computable $d_1$-name.
\end{lemma}
\begin{proof}
	The key idea is that since $U$ is $0$--$1$-valued a.e., the value $d_1(f_n, U)$ is computable.  
	This is because given a square in $\cP_n \times \cP_n$, the Lebesgue measure of those points $(x,y)$ in the square such that 
	$U(x,y) = 1$ is equal to the value that $f_n$ takes there, and this suffices to compute the $L_1$ norm of the difference on that square.
	For example, for the constant function $f_0$, if $f_0 \equiv p$, then
	\begin{align*}
		d_1(f_0, U) &= \lambda\{U=0\} \cdot p + \lambda\{U=1\} \cdot (1-p) \\
		& = (1-p)p + p(1-p) = 2p(1-p).
	\end{align*}

	Since $\seqn{f_n}$ is a $d_1$-name for $U$, and since we can compute each 
	quantity $d_1(f_n, U)$, we may 
	find a subsequence $\seqn{f_{k_n}}$ 
	such that $d_1(f_{k_n},  U) < 2^{-n}$. From this sequence
	$\seqn{f_{k_n}}$ we can find a computable $d_1$-name for $U$.
\end{proof}

This implies that there is a computable procedure for translating $d_1$-names to $\dcut$-names 
in the case of a random-free graphon.

\begin{theorem}
Let $U$ be a random-free graphon computable in $\dcut$.  Then $U$ is computable in $d_1$.
\end{theorem}
\begin{proof}
	Since $U$ is computable in $\dcut$, 
	by Theorem~\ref{martingale-comp-from-cut-norm}
	we can compute the martingale $\seqn{\UPn}$, which converges to $U$ in $d_1$.
	Then, because $U$ is random-free, we can compute a $d_1$-name for the limit of this martingale
	by 
	Lemma~\ref{martingale-comp-for-randomfree}.
\end{proof}

We have just seen that for random-free graphons, unlike the general case, $d_1$ and $\dcut$-names can be computably transformed into each other.
	One might therefore wonder whether one can computably determine that a graphon is random-free. In fact, it is not possible to recognize when a graphon is random-free, as we now demonstrate. On the other hand $0'$ does allow us to recognize this.

\begin{proposition}
The collection of $d_1$-names for random-free graphons is co-c.e.\ complete.
\end{proposition}
\begin{proof}
	Recall that $W(1-W) \ge 0$ a.e.\ for a graphon $W$, as it takes values in $[0,1]$. Also, $W$ is random-free if and only if $\int W(1-W)\, d\lambda = 0$. (See, e.g., \cite[Lemma~10.4]{MR3043217}.) Therefore,
	given a $d_1$-name for a graphon, we may compute $\int W(1-W)\, d\lambda$, and hence by noticing when this quantity is positive, we may enumerate the $d_1$-names of the non-random-free graphons.
Hence the $d_1$-names of random-free graphons are co-computably enumerable.

We now show that the collection of $d_1$-names of random-free graphons is complete. 
	For $s,e\in\Nats$, let $U^s_e$ be the constant function $2^{-s}$ if $\{e\}_s(0)\diverges$,
	and $U^s_e$ be the constant function $2^{-k}$ if $k\le s$ and minimal with $\{e\}_k(0)\converges$.

	Observe that for each $e\in\Nats$, the sequence $\<U^s_e\>_{s\in\Nats}$ is a computable $d_1$-name for a graphon that is not random-free if and only if $e\in 0'$.
\end{proof}
	
	In particular, there is no computer program that, given a $d_1$-name of a graphon, correctly asserts whether or not the graphon is random-free.

Having shown that every graphon with a computable $\dcut$-name has a $d_1$-name that is $0'$-computable, one may ask if this is tight, i.e., if $0'$ is necessary. We have just seen that this is not tight in the random-free case, and so any witness to the necessity of $0'$ must not be random-free. Next, in Section~\ref{lower-bound}, we provide such an example.

\section{$\dcut$-names vs.\ $d_1$-names: Lower bound}
\label{lower-bound}
We have just seen that using $0'$ we can compute a $d_1$-name of a graphon given a computable $\dcut$-name for it.
We now show that this is tight in the sense that there is a graphon that is computable in $\dcut$ such that $0'$ is computable from any $d_1$-name for a graphon weakly isomorphic to it.
Furthermore, we may take this graphon to be a.e.\ continuous.

\begin{theorem}
	\label{negative-AH}
There is an a.e.\ continuous graphon $U$ that is computable in $\dcut$, such that
		if $V$ is weakly isomorphic to $U$ then
		any $d_1$-name for
		$V$ computes the halting problem $0'$.
\end{theorem}

\begin{proof}
	For each $n\in\Nats$, define the open interval $A_n \defas (1 - 2^{-n}, 1 -2^{-(n+1)})$. The graphon $U$ will take the value zero outside the block-diagonal $\bigcup_{n\in\Nats} (A_n \times A_n)$. 
	Also, for each $n\in\Nats$, let $\ell_n \defas 1- 2^{-(2n +1)}$, let 
	$r_n \defas 1 - 2^{-(2n+2)}$, and let
	$m_n \defas \frac{\ell_n + r_n}2$,
	so that
	\[
		0 < \ell_0 < m_0 < r_0 < \ell_1 < m_1 < r_1 < \ldots < 1.\]
	
	Define the constant graphon $H^*\equiv\frac12$.
	Because $H^*$ is computable in $\dcut$, 
	for each $s\in\Nats$ we can computably find a
	random-free step function graphon 
	$G^*_s\in\StepGraphons$ such that $\normcut{H^* - G^*_s} < 2^{-s}$.
	For each $e\in\Nats$, let $\iota_e \colon [0,1] \to [\ell_e, r_e]$ be the unique increasing linear bijection. Note that $\iota_e(\frac12) = m_e$.
	Let $G_{e,s} \defas \iota_e \circ  G^*_s$ and
	$H_{e} \defas \iota_e \circ  H^*$.
	Observe that 
	\[
		\normcut{H_e - G_{e,s}} < \normcut{H^* - G^*_s} < 2^{-s}.
		\]

	Now for $s\in\Nats$, define $U_s$ to be the graphon that is $0$ outside of 
	$\bigcup_{e \le s} (A_e \times A_e)$ and for each $e\le s$ is equal to the 
	scaling to fit $A_e\times A_e$ of the
	following graphon $K_{e,s}$ on $[0,1]\times[0,1]$:
	\[
		K_{e,s} \defas \begin{cases}
			H_e & \text{if~} \{e\}_s(0)\diverges, \text{~and}\\
			G_{e,t} & \text{if~} t \le s \text{~is minimal such that~}  \{e\}_t(0)\converges.
		\end{cases}
		\]
	Note that 
	\begin{align*}
		\normcut{U_s - U_{s+1}} 
		&\le 2^{-s} + \sum_{n> s} \lambda(A_n \times A_n) \\
		&< 2^{-s} + 2^{-(s+1)} \cdot 2^{-(s+1)}\\
		&< 2^{-s+1}.
	\end{align*}
		Clearly the graphons $K_{e,s}$ are uniformly computable in $\dcut$, and so the sequence $\<U_s\>_{s\in\Nats}$ is a computable $\dcut$-name.
		For $x,y\in[0,1]$, define 
		\[
		U(x,y) = \lim_{s\to\infty} U_s(x,y)
		\]
		when it is defined, and $0$ otherwise.
		Note that $U$ is a limit of the sequence 
	$\<U_s\>_{s\in\Nats}$, and it
	is a.e.\ continuous, as it is piecewise constant (i.e., a step function with countably many steps).

	Note that for each $n\in\Nats$, we have 
	\[
		\lambda(U^{-1}\{\ell_n, m_n, r_n\}) = \lambda(A_n \times A_n) = 2^{-2(n+1)}
		\]
		by construction, as each $K_{n,s}$ only takes values among $\ell_n$, $m_n$, or $r_n$. 
	For any graphon $W$ 
	define
	the set 
	\[
		X_W \defas \bigl\{ x \in [0,1] \st \lambda\bigl(W^{-1}(x)\bigr) > 0\bigr\}.
		\]
	Note that for any $V$
weakly isomorphic to $U$, by
	Theorem~\ref{weakly-isomorphic-defthm} condition (3) we have $X_V = X_U$.
	Further, as 
	\[
		\{\ell_n, m_n, r_n\} \cap \{\ell_p, m_p, r_p\} = \emptyset
		\]
		for $p\ne n$, from a $d_1$-name for $V$ we can compute the countable discrete set $X_U$.
	But $m_e\in X$ if and only if $\{e \}(0)\diverges$, and so $X_U$ computes $0'$.
\end{proof}

\section{Almost-everywhere continuity}
\label{ae-cont-sec}

In this section, we describe a random-free graphon that is computable in $d_1$ but not weakly isomorphic to any a.e.\ continuous graphon.
Note that this is in contrast to the computable de~Finetti theorem \cite{MR2880271}, which can be seen as saying that in a $1$-dimensional analogue of this setting, the measurable object representing the sampler
is a.e.\ computable, and in particular a.e.\ continuous.
This provides another example of how the 2-dimensional case is considerably more complicated than the 1-dimensional case.

The notion of a.e.\ continuity is sensitive to the underlying topology of the space. Since a graphon is a function from $[0,1] \times [0,1]$ to $[0,1]$, it is reasonable to consider, as the topology on the domain, the product topology with respect to the usual topology on $[0,1]$. But there are situations where it is natural to consider graphons that are a.e.\ continuous with respect to other topologies on $[0,1]$ but are not weakly isomorphic to an a.e.\ continuous graphon with respect to the standard topology on $[0,1]$. We show that our result holds even in these more general situations, as long as the topology on $[0,1]$ still generates the Borel sets.

We begin by describing the construction of a random-free graphon $G$, which
can be thought of
as a symmetric measurable subset of $[0,1]^2$. We have drawn this measurable subset in Figure~\ref{fractal}
as a black ($1$) and white ($0$) picture, with $(0,0)$ in the upper-left corner 
(similar to an adjacency matrix, and as is common when drawing graphons). 

\begin{construction}
	First, draw a $2 \times2$ square grid (given by products of the intervals $[0,\frac12)$ and $[\frac12, 1]$ on each axis) and color the 2 squares on 
the diagonal black. Then on each of the 2 off-diagonal squares, draw
	a $4\times 4$ square grid (similarly, from products of half-open or closed intervals) and color the 8 diagonal squares black. Then on each of the remaining
24 squares, draw an $8\times 8$ grid and color the diagonal black. Continue in this way, filling diagonal squares within unfilled squares
	to obtain the graphon $G$.
\end{construction}

This clearly describes a symmetric measurable subset of $[0,1]^2$, and hence a random-free graphon.

\begin{figure}[h]
\includegraphics[scale=0.7]{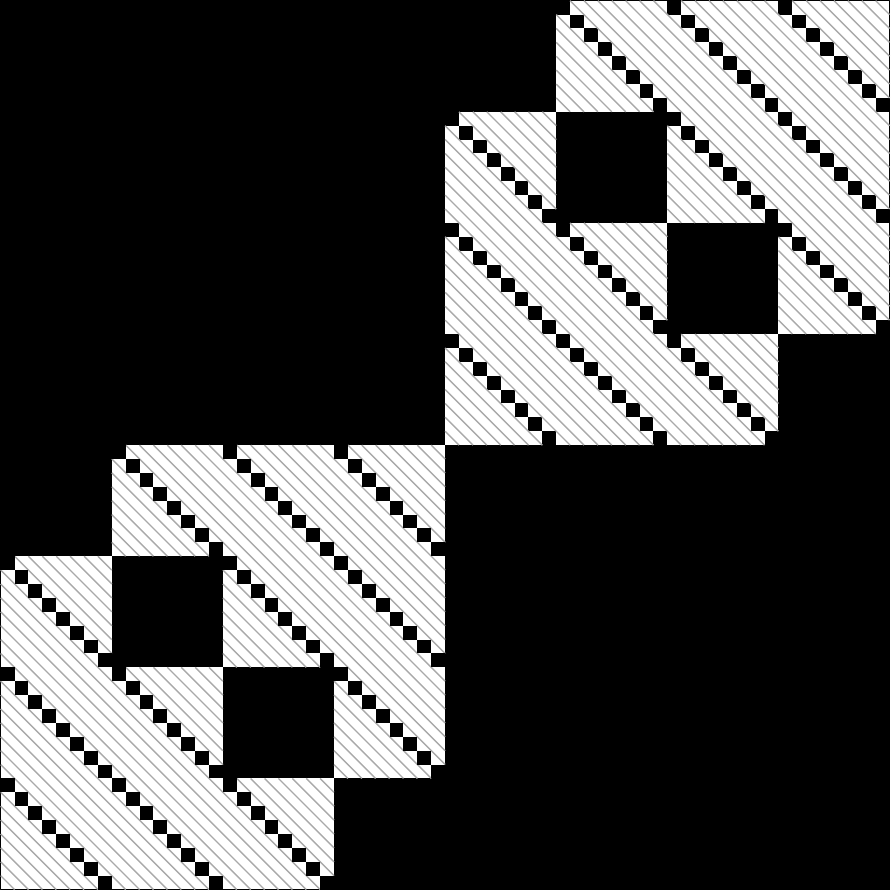}\\
	\caption{The graphon $G$, viewed as a subset of $[0,1]^2$.}
	\label{fractal}
\end{figure}

The countable random graph induced by sampling from $G$ may be thought of informally in the following way, which shows that it can be sampled in polynomial time:
There is a questionnaire with an infinite list of questions indexed by positive integers. The $n$th question has $2^n$ possible answers. 
Each vertex corresponds to a person who has independently answered each question uniformly at random, independently from each other person.
Two vertices are connected by an edge when the corresponding people agree on at least one answer to their questionnaires.

\begin{lemma} The graphon $G$ constructed above is 
	not a.e.\ continuous, has a computable $d_1$-name, and is twin-free.
\end{lemma}
\begin{proof}
	The black region $G^{-1}(\{1\})$ is clearly dense. Also the white region $G^{-1}(\{0\})$ has measure equal to $\alpha \defas \frac{1}{2}\cdot\frac{3}{4}\cdot\frac{7}{8}\cdots$
which is positive by the following calculation:
	\begin{align*}
-\log(\frac{1}{2}\cdot\frac{3}{4}\cdot\frac{7}{8}\cdots) 
		&=-\sum_{n=1}^{\infty}\log(1-2^{-n})\\
		&=\sum_{n=1}^{\infty}\sum_{k=1}^{\infty}\frac{2^{-kn}}{k}\\
		&=\sum_{k=1}^{\infty}\frac{1}{k}\sum_{n=1}^{\infty}2^{-kn}\\
		&=\sum_{k=1}^{\infty}\frac{1}{k}\frac{1}{2^{k}-1}<\infty.
	\end{align*}
So this graphon itself is not a.e.\ continuous, since 
	$G^{-1}(\{0\})$
	is a nowhere dense set of positive measure.

But $G$ has a computable $d_1$-name since one can approximate the
graphon in $d_1$ with some initial stage of the construction, as we now describe.
	Let $G_n$ be the $n$th stage of the construction, and let $\beta_n \defas \int G_n \, d\lambda$. Then the measure of the remaining black area yet to be added to $G$ is $(1-\alpha) - \beta_n$, a computable real that rapidly tends to $0$ as $n\to\infty$.

Finally, observe that $G$ is twin-free, 
	as by construction, each horizontal line gives
rise to a different cross-section.
\end{proof}

The main result of this section is that no graphon $H$ weakly isomorphic to $G$ is a.e.\ continuous (even with respect to other topologies that generate the Borel sets). The key combinatorial fact is the following.

\begin{lemma}
	\label{product-measure-zero}
	Let $G$ be the random-free graphon constructed above.
	Suppose $X, Y \subseteq [0,1]$ are measurable sets such that
	$X\times Y$ is contained in $G^{-1}(\{0\})$ up to a nullset.
	Then $\lambda(X\times Y) = 0$.
\end{lemma}

\begin{proof}
	For each $n\in\Nats$, let $X_n$ (respectively $Y_n$) be the union of all dyadic half-open intervals 
	of size $2^{-(2^n - 1)}$ whose intersection with $X$ (respectively $Y$) has positive measure.  
	To show that $\lambda(X \times Y) = 0$, we will show by induction that $\lambda(X_n \times Y_n) \leq 4^{-n}$.  
	
	The base case is trivial as $\lambda(X_0 \times Y_0) \leq 1$.	For the induction step, consider each dyadic square $I \times J$ where $I \subseteq X_n$ and $J \subseteq Y_n$ are both of size $2^{-(2^n - 1)}$.  By construction, for each sub-dyadic interval $I' \subseteq I$ of size $2^{-(2^{n+1} - 1)}$, there is a corresponding dyadic interval $J' \subseteq J$ of the same size such that $I' \times J'$ is a black square.  If $I'$ is disjoint from $X$ (up to a nullset), then $I' \subseteq X_n \setminus X_{n+1}$.  Otherwise, $J'$ is disjoint from $Y$ (up to a nullset), and $J' \subseteq Y_n \setminus Y_{n+1}$;  for if not, then the black square $I' \times J'$ intersects $X \times Y$ outside a nullset, which cannot happen, since $X \times Y \subseteq G^{-1}(\{0\})$ is white. 
	After considering all such sub-dyadic intervals $I'$ we have that 
	\[\lambda(X_{n+1} \cap I) + \lambda(Y_{n+1} \cap J) \leq \frac{\lambda(I) +\lambda(J)}{2} = \lambda(I).\]
	By the arithmetic--geometric mean inequality,
	\begin{align*}
		\lambda\bigl((X_{n+1} \times Y_{n+1}) \cap (I\times J) \bigr)
	&= \lambda(X_{n+1} \cap I) \cdot \lambda(Y_{n+1} \cap J) \\
	&\leq \left(\frac{\lambda(X_{n+1} \cap I) + \lambda(Y_{n+1} \cap J)}{2}\right)^2 \\
	&\leq \left(\frac{\lambda(I)}{2}\right)^2 = \frac{\lambda(I)^2}{4}.
	\end{align*}
	Summing up over all such $I \times J$ and using the induction hypothesis we have  
	\[\lambda(X_{n+1} \times Y_{n+1}) \leq \frac{\lambda(X_n \times Y_n)}{4} \leq 4^{-(n+1)}.\]
	Therefore $\lambda(X \times Y) = 0$. 
\end{proof}

We may now prove the main result about $G$.

\begin{theorem}
	\label{not-ae-continuous}
	Let $G$ be the random-free graphon (which has a computable $d_1$-name) constructed above.
	Let $H$ be a graphon 
	weakly isomorphic to $G$,
	and let $\cT$ be some topology on $[0,1]$ every open set of which is a standard Borel set.
	Then $H$ is not a.e.\ continuous with respect to $\cT\times\cT$.
\end{theorem}

\begin{proof}
	Because $G$ 
	is twin-free and weakly isomorphic to $H$,
	by \cite[Theorem~8.6 (vi)]{MR3043217}
	there is a measure-preserving map $\psi\colon [0,1] \to [0,1]$ such that
	$H = G^\psi$ a.e.
	Hence $H$ must be random-free as well.
	
	Now assume, towards a contradiction, that the map 
	$H$ is a.e.\ continuous with respect to $\cT\times\cT$.
	Because $\psi$ is measure-preserving, 
	\[
		\lambda\bigl(G^{-1}(\{0\})\bigr) = \lambda\bigl(H^{-1}(\{0\})\bigr).
		\]
	Define $A\subseteq[0,1]^2$ to be the set $H^{-1}(\{0\})$.

	Since $H$ is a.e.\ continuous with respect to $\cT \times \cT$,
	we have
	that $A$ is a $\lambda$-continuity set, and therefore its interior (in the
	product topology $\cT \times \cT$) is standard Borel and has the same (positive)
	measure as $A$. Hence there is some open set $B\times C\subseteq A$
	where $B$ and $C$ are $\cT$-open sets of $[0,1]$ (and hence are standard Borel sets)
	that have positive measure.
	
	Let $\lambda(\cdot|B)$ denote Lebesgue measure conditioned on $B$, that is 
	\[
		\lambda(A|B) = \frac{\lambda(A \cap B)}{\lambda(B)}.
		\]
		Both
	$\lambda(\cdot |B)$ and $\lambda(\cdot|C)$ are well-defined since $B$ and $C$ have positive measure.  
	Now, let $\mu_B$ and $\mu_C$ denote the pushforward measures on $[0,1]$ of $\lambda(\cdot|B)$ and $\lambda(\cdot|C)$ along the map $\psi$.
	That is, $\mu_B(S) = \lambda(\psi \in S | B)$ for all measurable $S \subseteq [0,1]$, and likewise with $C$.  
	We claim that $\mu_B$ and $\mu_C$ are absolutely continuous with respect to $\lambda$.  
	Indeed, if $\lambda(S) = 0$, then 
	\[\mu_B(S) = \lambda(\psi \in S|B) \leq \frac{\lambda(\psi \in S)}{\lambda(B)} = \frac{\lambda(S)}{\lambda(B)} = 0.\]
	Therefore, the supports of $\mu_B$ and $\mu_A$ have positive $\lambda$-measure. 
	
	Because $\psi$ is a measure-preserving map, we have
	\begin{align*}
		(\mu_B \otimes \mu_C) \{G=0\} 
		&= \frac{1}{\lambda(A)\lambda(B)} \int_C \int_B \Ind{\{G^\psi=0\}} \, d\lambda\,  d\lambda  \\
		&= \frac{1}{\lambda(A)\lambda(B)} \int_C \int_B \Ind{\{H=0\}} \, d\lambda \, d\lambda \\
	&= 1, 
	\end{align*}
	where the last equality follows from the fact that $H = 0$ on $B \times C$.
	
	Now let $X$ be the support of $\mu_B$ and $Y$ be the support of $\mu_C$. Then $X\times Y$ is 
	contained in $G^{-1}(\{0\})$ up to a nullset.
	By Lemma~\ref{product-measure-zero}, we have
	$\lambda(X\times Y)=0$, and so
	one of $X$ and $Y$ has measure $0$, a contradiction to the fact that $\mu_B$ and $\mu_C$ are absolutely continuous probability measures. 
	Hence $H$ is not a.e.\ continuous with respect to $\cT \times \cT$.
\end{proof}

We note that by taking the direct sum of this example with the lower bound construction of Section~\ref{lower-bound}, we may obtain a $0'$-computable $d_1$-name for a graphon for which no weakly isomorphic graphon is a.e.\ continuous and from which any $d_1$-name computes $0'$. 
Namely, scale the graphon from Theorem~\ref{negative-AH}
by $\frac12$ and place it on $[0,\frac 12]^2$, along with $G$ scaled by $\frac12$ on $[\frac12, 1]^2$.


\newpage
\begin{small}
\section*{Acknowledgments}
The authors would like to thank
Jan Reimann and
Ted Slaman
for helpful discussions.
Work on
this publication was made possible through the support of 
United States Air Force Office of Scientific Research (AFOSR)
Contract No.\ FA9550-15-1-0074.
J.A.\ was supported by 
AFOSR
MURI award FA9550-15-1-0053 and National Science Foundation (NSF) grant DMS-1615444.
C.E.F.\ was supported by 
United States Air Force (USAF) and the Defense Advanced Research Projects Agency
(DARPA) Contracts No.\ FA8750-14-C-0001 and FA8750-14-2-0004, Army Research Office (ARO) grant W911NF-13-1-0212, Office of Naval Research (ONR) grant N00014-13-1-0333, 
NSF grants DMS-0800198 and DMS-0901020, and grants from the John Templeton Foundation and Google.
D.M.R.\ was supported by a Newton International Fellowship, Emmanuel Research Fellowship, NSERC Discovery Grant, and Connaught Award.
Any opinions, findings and conclusions or recommendations expressed in this material are those of the authors and do not necessarily reflect the views of the United States Air Force, Army, Navy, DARPA, or the John Templeton Foundation.
\end{small}

\vspace*{-3pt}
\begin{small}

\newcommand{\etalchar}[1]{$^{#1}$}
\providecommand{\bysame}{\leavevmode\hbox to3em{\hrulefill}\thinspace}
\providecommand{\MR}{\relax\ifhmode\unskip\space\fi MR }
\providecommand{\MRhref}[2]{%
  \href{http://www.ams.org/mathscinet-getitem?mr=#1}{#2}
}
\providecommand{\href}[2]{#2}

\end{small}


\end{document}